\documentclass[a4paper,10pt]{amsart}

\usepackage{latexsym} \usepackage{mathrsfs}

\usepackage{amsmath} \usepackage{amssymb} \usepackage{amsfonts}
\usepackage{amsthm} \usepackage{amsxtra} 

\usepackage{hyperref}

\usepackage{enumerate}

\usepackage{epsfig} \input{xy} \xyoption{all}

\newcommand{\abs}[1]{\left\lvert{#1}\right\rvert}
\newcommand{\norm}[1]{\left\|{#1}\right\|}

\newcommand{\R}{\mathbb{R}}\newcommand{\N}{\mathbb{N}}

\newcommand{\A}{\mathcal{A}}

 \newcommand{\ie}{i.e.\ }

\newtheorem*{mainthm}{Theorem A} 
\newtheorem{theorem}{Theorem}[section]

\newtheorem{lemma}[theorem]{Lemma}

\theoremstyle{definition}

\theoremstyle{remark} 

\newcommand{\Dis}[1]{\mathcal{D}^\prime_{#1}}

 \newcommand{\dd}{\:\mathrm{d}}

\DeclareMathOperator{\M}{\mathfrak{M}}

\DeclareMathOperator{\Bs}{\mathcal{S}}
\DeclareMathOperator{\Lf}{\mathcal{L}_\mathit{f}}

\begin{document}

\title{On cohomological $C^0$-(in)stability}

\author{Alejandro Kocsard} \address{Instituto de Matem\'atica e
  Estat\'\i stica, Universidade Federal Fluminense, Rua M\'ario Santos
  Braga, s/n. Niter\'oi, RJ, Brazil.} \email{akocsard@id.uff.br}

\date{\today}

\begin{abstract}
  After Katok\cite{KatokRobinson}, a homeomorphism $f\colon M\to M$ is
  said to be \emph{cohomologically $C^0$-stable} when its space of
  real $C^0$-coboundaries is closed in $C^0(M)$. In this short note we
  completely classify cohomologically $C^0$-stable homeomorphisms,
  showing that periodic homeomorphisms are the only ones.
\end{abstract}

\maketitle

\section{Introduction}
\label{sec:intro}

Cocycles and cohomological equations play a fundamental role in
dynamical systems and ergodic theory. In this short note we shall
mainly concentrate on topological dynamics. So, from now on $(M,d)$
will denote a compact metric space and the dynamics will be given by a
homeomorphism $f\colon M\to M$.

In such a case, a (real) \emph{cocycle over $f$} is just a map
$\phi\colon M\to\R$. If $\A\subset M^\R$ denotes an
\emph{$f$-invariant functional space} (\ie $\A$ is linear subspace of
$M^\R$ such that $\psi\circ f\in\A$ whenever $\psi\in\A$), any
$\phi\in\A$ will be called an \emph{$\A$-cocycle} and we will say
$\phi$ is an \emph{$\A$-coboundary} whenever the cohomological
equation
\begin{displaymath}
  \phi=u\circ f-u
\end{displaymath}
admits a solution $u\in\A$. Many questions in dynamics can be reduced
to determine if certain cocycles are or not coboundaries, so it is an
important problem (and in many cases rather difficult) to study the
structure of the linear space of $\A$-coboundaries, which shall be
denoted by
\begin{displaymath}
  B(f,\A):=\{v\circ f-v : v\in\A\}.
\end{displaymath}

By analogy with cohomological theories, we can define the \emph{first
  cohomology space of $f$} with coefficient in $\A$ as the linear
space
\begin{displaymath}
  H^1(f,\A):=\A/B(f,\A).
\end{displaymath}

To analyze the structure of $H^1(f,\A)$ (and $B(f,\A)$) in general we
endow the space $\A$ with a vector space topology, and hence,
$H^1(f,\A)$ inherits the quotient one. Then, typically the analysis is
divided in two steps (see \cite{KatokRobinson} for a very detailed
exposition):
\begin{enumerate}[(a)]
\item{\textbf{Cohomological $\A$-obstructions:}} very roughly, these
  are necessary conditions an $\A$-cocycle must satisfy to be an
  $\A$-coboundary. In general, these are closed conditions in $\A$, so
  typically they characterize $\overline{B(f,\A)}^\A$ instead of
  $B(f,\A)$. Some examples of cohomological obstructions:
  \begin{enumerate}[(i)]
  \item Invariant measures are the cohomological obstructions for
    solving cohomological equations in the topological category. In
    fact, if $\M(f)$ denotes the space of $f$-invariant probability
    measures, then it holds
    \begin{displaymath}
      \overline{B(f,C^0(M))}^{C^0}=\left\{\phi\in C^0(M) :
        \int_M\phi\dd\mu=0,\ \forall\mu\in\M(f)\right\}.
    \end{displaymath}
  \item Invariant distributions (in the sense of Schwartz) are the
    cohomological obstructions in the smooth category. In fact, if $M$
    is closed smooth manifold and $\Dis{}(M)$ denotes the topological
    dual space of $C^\infty(M)$, then defining
    $\Dis{}(f):=\{\mu\in\Dis{}(M) :
    \langle\mu,\phi\rangle=\langle\mu,\phi\circ f\rangle,\
    \forall\phi\in C^\infty(M)\}$, it holds
    \begin{displaymath}
      \overline{B(f,C^\infty(M))}^{C^\infty}=\left\{\phi\in
        C^\infty(M) : \langle\mu,\phi\rangle=0,\
        \forall\mu\in\Dis{}(f)\right\}.
    \end{displaymath}
  \end{enumerate}
\item \textbf{Cohomological $\A$-stability:} A system $f$ is said to
  be \emph{cohomologically $\A$-stable} when $B(f,\A)$ is closed in
  $\A$. Cohomological stability is a very desirable property because
  in that case, and only in that case, we can verify whether a cocycle
  $\phi$ is an $\A$-coboundary just analyzing the cohomological
  obstructions of item (a). Let us mention some examples:
  \begin{enumerate}[(i)]
  \item\emph{Hyperbolic systems:} After Liv\v sic
    \cite{LivsicCohomDynSys} we know that hyperbolic systems are
    cohomologically H\"older-stable. On the other hand, de la Llave,
    Marco and Moriyon have shown in \cite{delaLlaveMarcoMoriyon} that
    $C^r$ Anosov diffeomorphisms are cohomologically $C^r$-stable, for
    any $r\in [2,\infty]$.
  \item\emph{Ergodic translations on tori:} It is well-known that
    ergodic translations on tori are cohomologically $C^\infty$-rigid
    iff they are Diophantine (see \cite{KatokRobinson} for details). 
  \item\emph{Smooth circle diffeomorphisms with irrational rotation
      number:} In a joint work with Avila
    \cite{AviKocCohomoEqInvDistCirc}, we showed that a
    $C^\infty$-circle diffeomorphism with no periodic points is
    cohomologically $C^\infty$-stable iff its rotation number is
    Diophantine.
  \end{enumerate}
\end{enumerate}

In this short note, we completely characterize the homeomorphisms that
are cohomologically $C^0$-stable. In fact, we prove the following
\begin{mainthm}
  A homeomorphism $f\colon M\to M$ is cohomologically $C^0$-stable if
  and only if $f$ is periodic, \ie it has finite order in the group of
  homeomorphisms of $M$.
\end{mainthm}

\section{Notations}
\label{sec:notations}

As we have already mentioned in the \S~\ref{sec:intro}, $(M,d)$ will
denote an arbitrary compact metric space. Given $x\in M$ and $r>0$, we
write $B(x,r):=\{y\in M : d(x,y)<r\}$. If $A\subset M$, $\chi_A\colon
M\to\{0,1\}$ will denote the characteristic function of $A$. 

We will write $C^0(M)$ for the space of real continuous functions on
$M$ endowed with the uniform norm
\begin{displaymath}
  \norm{\phi}_{C^0}:=\sup_{x\in M}\abs{\phi(x)},\quad\forall \phi\in
  C^0(M). 
\end{displaymath}

Given a homeomorphism $f\colon M\to M$, we define the space of
\emph{$C^0$-coboundaries} by
\begin{displaymath}
  B(f,C^0(M)):=\{v\circ f-v :  v\in C^0(M)\}. 
\end{displaymath}

The homeomorphism $f$ is said to be \emph{cohomologically
  $C^0$-stable} iff $B(f,C^0(M))$ is closed in $C^0(M)$.

On the other hand, $f$ is said to be \emph{periodic} when there exists
$q\in\N$ satisfying $f^q=id_M$, and the number $q$ is called a
\emph{period of $f$.}

\section{Proof of Theorem A}
\label{sec:proof-theorem-A}

Let us start with the simplest part of Theorem $A$, \ie let us prove
that any periodic map is cohomologically $C^0$-stable:

\begin{lemma}
  \label{lem:per-impl-stable}
  Let us assume $f$ is periodic and let $q\in\N$ be a period of
  $f$. Then, 
  \begin{displaymath}
    B(f,C^0(M))=\Bigg\{\phi\in C^0(M) :
    \sum_{j=0}^{q-1}\phi(f^j(x))=0,\ \forall x\in M\Bigg\}. 
  \end{displaymath}
  In particular, $f$ is cohomologically $C^0$-stable.
\end{lemma}

\begin{proof}
  First of all observe that every $\phi\in B(f,C^0(M))$ satisfies
  $\Bs_f^q\phi\equiv 0$. In fact, if $u\colon M\to\R$ is such that
  \begin{displaymath}
    \phi(x)=u(f(x))-u(x),\quad\forall x\in M, 
  \end{displaymath}
  then it clearly holds 
  \begin{displaymath}
    \Bs_f^q\phi(x)=u(f^q(x))-u(x)=0,\quad\forall x\in M.
  \end{displaymath}

  On the other hand, let us suppose $\psi\in C^0(M)$ is such that
  $\Bs_f^q\psi\equiv 0$. Then, using a formula we learned from
  \cite{MoulinPinchonSyst} we write
  \begin{displaymath}
    v(x):=-\frac{1}{q}\sum_{j=1}^q\Bs_f^j\psi(x),\quad \forall
    x\in M.
  \end{displaymath}
  It clearly holds $v\in C^0(M)$, and
  \begin{displaymath}
    \begin{split}
      v(f(x))-v(x)&=-\frac{1}{q}\bigg(\sum_{j=1}^q\big(
      \Bs_f^j\psi(f(x))-\Bs_f^j\psi(x)\big)\bigg) \\
      & = -\frac{1}{q}\bigg(\Bs_f^q\psi(f(x))-q\psi(x)\bigg)=\psi(x),
    \end{split}
  \end{displaymath}
  for every $x\in M$. Thus, $\psi\in B(f,C^0(M))$, as desired.
\end{proof}

\subsection{The cohomological operator}
\label{sec:cohomol-operator}

We can define the \textbf{cohomological operator} (associated to
$f$) $\Lf\colon C^0(M)\to C^0(M)$ by
\begin{displaymath}
  \Lf(u):=u\circ f-u\quad\forall u\in C^0(M).
\end{displaymath}
This is clearly a linear operator, and since
\begin{displaymath}
  \norm{\Lf(u)}_{C^0}\leq 2\norm{u}_{C^0},
\end{displaymath}
it is is also continuous.

Now, observe that the kernel of $\Lf$, which shall be denoted by
$\ker\Lf$, coincides with the space of continuous $f$-invariant
functions. The quotient space $C^0(M)/\ker\Lf$ will be denoted by
$C^0_f(M)$. Defining
\begin{equation}
  \label{eq:norm-C0f-def}
  \norm{\phi+\ker\Lf}_{C^0_f}:=\inf_{\psi\in\ker\Lf}
  \norm{\phi+\psi}_{C^0}, \quad\forall\phi\in C^0(M),
\end{equation}
we clearly get a norm and this turns $C^0_f(M)$ into a Banach space.

On the other hand, notice that the image of the operator $\Lf$
coincides with the space of continuous coboundaries $B(f,C^0(M))$,
which is, by our hypothesis, a closed subspace of $C^0(M)$. We will
consider $B(f,C^0(M))$ equipped with (the restriction of) the norm
$\norm{\cdot}_{C^0}$.

In this way, we have the following simple 
\begin{lemma}
  \label{lem:induced-operat-cont}
  Let $\bar\Lf\colon C^0_f(M)\to B(f,C^0(M))$ be the factor linear
  operator turning the following diagram commutative:
  \begin{displaymath}
    \xymatrix{C^0(M)\ar[rr]^\Lf\ar[dr]^{\pi} & & B(f,C^0(M)) \\
      & C^0_f(M)\ar[ur]^{\bar\Lf} &}
  \end{displaymath}
  where $\pi : \phi\mapsto\phi+\ker\Lf$ denotes the canonical
  quotient projection.

  Then, $\bar\Lf$ is continuous and bijective, and consequently, it is
  a Banach space isomorphism.
\end{lemma}

\begin{proof}
  The continuity of $\bar\Lf$ easily follows from the following
  estimate: for any $\phi\in C^0(M)$ and every $\psi\in\ker\Lf$, it
  holds
  \begin{displaymath}
    \norm{\bar\Lf(\phi+\ker\Lf)}_{C^0}=\norm{\Lf(\phi+\psi)}_{C^0}\leq
    2\norm{\phi+\psi}_{C^0}.
  \end{displaymath}
  Taking infimum over $\psi\in\ker\Lf$ on the right hand side, we get
  \begin{displaymath}
    \norm{\bar\Lf(\phi+\ker\Lf)}_{C^0}\leq
    2\norm{\phi+\ker\Lf}_{C^0_f}, \quad\forall\phi\in C^0(M). 
  \end{displaymath}

  Finally, since $\bar\Lf$ is tautologically bijective, by the open
  mapping theorem, $\bar\Lf$ is a Banach space isomorphism.
\end{proof}

Now, in order to finish the proof of Theorem~A, let us assume $f$ is
cohomologically $C^0$-stable and it is not periodic. That means for
every $n\in\N$, we can find $x_n\in M$ such that $x_n\neq f^j(x_n)$,
for every $j\in\{1,\ldots,2^n\}$.

For each $n\geq 1$, let us choose $r_n>0$ such that the ball
$B_n:=B(x_n,r_n)$ satisfies
\begin{displaymath}
  f^j(B_n)\cap B_n=\emptyset,\quad\forall
  j\in\{1,\ldots, 2^n\}. 
\end{displaymath}

Then, consider the function $u_n\colon M\to\R$ given by
\begin{equation}
  \label{eq:u_n-def}
  u_n(x):=\sum_{j=-2^n+1}^{2^n-1}
  \chi_{f^j(B_n)}(x)\bigg(1-\frac{\abs{j}}{2^n}\bigg)
  \frac{r_n - d\big(f^{-j}(x),x_n\big)}{r_n}, 
  \quad\forall x\in M. 
\end{equation}
One can easily check that $u_n$ is continuous, its support is equal to
the disjoint union $\bigsqcup_{\abs{j}<2^n} f^j(B_n)$,
\begin{equation}
  \label{eq:u_n-inf}
  \inf_{x\in M} u_n(x)=u_n\Big(f^{2^n}(x_n)\Big)=0,
\end{equation}
and 
\begin{equation}
  \label{eq:u_n-sup}
  \sup_{x\in M} u_n(x)=u_n(x_n)=1.
\end{equation}

Since any function $v\in\ker\Lf$ must satisfy
$v(x_n)=v\big(f^{2^n}(x_n)\big)$, from \eqref{eq:u_n-inf} and
\eqref{eq:u_n-sup} we conclude that
\begin{equation}
  \label{eq:u_n-C0-quot-norm}
  \norm{u_n+\ker\Lf}_{C^0_f} \geq\frac{1}{2}.
\end{equation}

Now, consider the coboundary $\phi_n:=\Lf(u_n)=u_n\circ f-u_n\in
B(f,C^0(M))$. Thus, for every $x\in M$ it holds
\begin{equation}
  \label{eq:phi_n-formula}
  \begin{split}
    \phi_n(x)&=u_n(f(x))-u_n(x) \\
    & =\sum_{j=-2^n+1}^{2^n-1} \chi_{f^j(B_n)}(f(x))
    \bigg(1-\frac{\abs{j}}{2^n}\bigg)
    \frac{r_n - d\big(f^{-j+1}(x),x_n\big)}{r_n}  \\
    &\qquad\qquad\quad - \sum_{j=-n+1}^{n-1}\chi_{f^j(B_n)}(x)
    \bigg(1-\frac{\abs{j}}{2^n}\bigg)
    \frac{r_n-d(f^{-j}(x),x_n)}{r_n}\\
    &=\Bigg[\chi_{f^{-2^n}(B_n)}(x)
    \frac{r_n-d\big(f^{2^n}(x),x_n\big)}{2^nr_n} \\
    &\qquad+ \sum_{j=-2^n+2}^{2^n-2} \chi_{f^j(B_n)}(x)
    \left(\frac{\abs{j}}{2^n} -\frac{\abs{j+1}}{2^n}\right)
    \frac{r_n-d\big(f^{-j}(x),x_n\big)}{r_n}\\
    &\qquad\qquad\quad -\chi_{f^{2^n-1}(B_n)}(x)
    \frac{r_n-d\big(f^{-2^n+1}(x),x_n\big)}{2^nr_n}\Bigg]
  \end{split}
\end{equation}
In particular, \eqref{eq:phi_n-formula} implies that 
\begin{equation}
  \label{eq:phi_n-C0-norm}
  \norm{\phi_n}_{C^0}=\abs{\phi_n(f^j(x_n))}= \frac{1}{2^n},
  \quad\forall j\in\{-2^n,\ldots,2^n-1\}.
\end{equation}

Finally, recalling that $\Lf(u_n)=\phi_n$, for every $n\in\N$, from
\eqref{eq:u_n-C0-quot-norm} and \eqref{eq:phi_n-C0-norm} it follows
that $\Lf^{-1}\colon B(f,C^0(M))\to C^0_f(M)$ is not continuous,
contradicting Lemma~\ref{lem:induced-operat-cont}, and Theorem~A is
proved. 

\bibliographystyle{amsalpha} \bibliography{base-biblio}

\end{document}